\documentclass[10pt,twocolumn]{article}

\usepackage[left=.4in, right=.4in, top=.9in, bottom=.9in]{geometry}
\usepackage{natbib}
\usepackage{amssymb,mathtools}
\usepackage[colorlinks,allcolors=blue]{hyperref}
\usepackage[small]{titlesec}
\usepackage{amsmath,amsthm,amsfonts,amssymb,mathtools,bbm}
\usepackage{setspace}
\usepackage{anyfontsize}

\AddToHookNext{shipout/before}{%
	\setbox\ShipoutBox=\hbox{\kern0.4cm\box\ShipoutBox}%
}

\theoremstyle{plain}
\newtheorem{theorem}{Theorem}
\newtheorem{lemma}{Lemma}

\makeatletter
\newcommand{\oset}[3][0.5ex]{
	\mathrel{\mathop{#3}\limits^{
			\vbox to#1{\kern0.2\ex@
				\hbox{$\scriptstyle#2$}\vss}}}}

\newcommand{\uset}[3][0ex]{%
	\mathrel{\mathop{#3}\limits^{
			\vbox to#1{\kern5.5\ex@
				\hbox{$\scriptstyle#2$}\vss}}}}

\newcommand{\eqd}{\oset{\mathrm{\scriptscriptstyle{D}\,}}{=}}
\newcommand{\eqas}{\oset{\mathrm{\scriptscriptstyle{a.s.}\,}}{=}}

\DeclareMathOperator{\ran}{ran}

\makeatother

	\author{{\normalsize Brendan K.\ Beare} \\ {\normalsize School of Economics} \\ {\normalsize University of Sydney} \\ {\normalsize \href{mailto:brendan.beare@sydney.edu.au}{brendan.beare@sydney.edu.au}} \and {\normalsize Tetsuya Kaji} \\ {\normalsize Booth School of Business} \\ {\normalsize University of Chicago} \\ {\normalsize \href{mailto:tkaji@chicagobooth.edu}{tkaji@chicagobooth.edu}}}

\begin{document}

	\title{Convergence in distribution of the P--P process in $L^1[0,1]$}
	
	\twocolumn[{
	\begin{@twocolumnfalse}
	\maketitle
	\center{Accepted for publication in \textit{Stat}.}
	\bigskip
	\begin{abstract}
		We show that the percentile--percentile (P--P) process constructed from an independent and identically distributed sample of pairs converges in distribution in $L^1[0,1]$ if and only if the associated P--P curve is absolutely continuous. When this condition holds, the limiting distribution is Gaussian and the process admits a valid bootstrap approximation.
		\bigskip
	\end{abstract}
\end{@twocolumnfalse}
}]
\section{Main result and discussion}\label{sec:intro}

Let $\{(X_i,Y_i)\}_{i=1}^\infty$ be an independent and identically distributed (iid) sequence of pairs of random variables. Let $F:\mathbb R\to[0,1]$ and $G:\mathbb R\to[0,1]$ be the marginal cumulative distribution functions (cdfs) for $X_i$ and $Y_i$ respectively. Let $Q:(0,1)\to\mathbb R$ be the quantile function (qf) for $Y_i$; that is,
\begin{equation*}
	Q(u)=\inf\{y\in\mathbb R:G(y)\geq u\}.
\end{equation*}	
For each $n\in\mathbb N$, define the empirical cdfs $F_n:\mathbb R\to[0,1]$ and $G_n:\mathbb R\to[0,1]$ by
\begin{equation*}
	F_n(x)=\frac{1}{n}\sum_{i=1}^n\mathbbm 1(X_i\leq x),\quad G_n(y)=\frac{1}{n}\sum_{i=1}^n\mathbbm 1(Y_i\leq y),
\end{equation*}
and define the empirical qf $Q_n:(0,1)\to\mathbb R$ by
\begin{equation*}
	Q_n(u)=\inf\{y\in\mathbb R:G_n(y)\geq u\}.
\end{equation*}	
Define $R:[0,1]\to[0,1]$ and $R_n:[0,1]\to[0,1]$ by
\[
R(u) = F(Q(u)) \,\,\, \text{and} \,\,\, R_n(u) = F_n(Q_n(u))\,\,\, \text{for } u \in (0,1),
\]
\[
\begin{alignedat}{3}
	R(0)   & = \lim_{u\downarrow 0} R(u), \qquad R_n(0) & = \lim_{u\downarrow 0} R_n(u),\\
	R(1)   & = \lim_{u\uparrow 1} R(u), \qquad R_n(1) & = \lim_{u\uparrow 1} R_n(u).
\end{alignedat}
\]
We refer to $R$, $R_n$ and $\sqrt{n}(R_n-R)$ as the P--P curve, P--P plot and P--P process, respectively. The P--P curve is also known as the ordinal dominance curve or{, up to conventional choices of orientation, the receiver operating characteristic (ROC) curve. ROC curves are widely used to assess the discriminatory performance of medical diagnostic procedures and prediction models; see, for example, \citet{OB18}. They are also used to evaluate forecasts of economically or managerially important events; see \citet{G09}.}

Our main result is the following theorem.

\begin{theorem}\label{thm:iff}
	$\sqrt{n}(R_n-R)$ converges in distribution in $L^1[0,1]$ if and only if $R$ is absolutely continuous.
\end{theorem}

Here $L^1[0,1]$ denotes the usual space of Lebesgue measurable and integrable real-valued functions on $[0,1]$, equipped with the norm $\lVert\cdot\rVert_1$. {Functions in $L^1[0,1]$ that are equal almost everywhere (a.e.) are identified. Throughout this article, convergence in distribution ($\rightsquigarrow$) in a metric space is meant in the Hoffmann--J{\o}rgensen sense, as in \citet[Definition 1.3.3]{VW23}. Since $L^1[0,1]$ is separable, convergence in distribution in this space may equivalently be understood in the classical sense of \citet[pp.~25--27]{B99}. The more general Hoffmann--J{\o}rgensen concept is needed for some of our arguments below involving nonseparable spaces.}

When $R$ is absolutely continuous, the limit in distribution of the P--P process in $L^1[0,1]$ is a centered Gaussian process $\mathcal R:[0,1]\to\mathbb R$ whose covariance structure depends on $R$ and on the dependence between the paired observations. Let $C:[0,1]^2\to[0,1]$ be a copula for $(X_i,Y_i)$; that is, a bivariate cdf on $[0,1]^2$ with uniform margins satisfying
\begin{equation*}
	C(F(x),G(y))=\mathbb P(X_i\leq x,Y_i\leq y).
\end{equation*}
The copula $C$ is uniquely determined on $\overline{\ran(F)}\times\overline{\ran(G)}$, by Sklar's theorem (\citealp{S59}; see also \citealp{G24}). Let $\mathcal B:[0,1]^2\to\mathbb R$ be a centered Gaussian process with covariance kernel
\begin{equation*}
	\mathrm{Cov}(\mathcal B(u,v),\mathcal B(u',v'))=C(u\wedge u',v\wedge v')-C(u,v)C(u',v').
\end{equation*}
The process $\mathcal B$ is called a tied-down Brownian sheet with intensity measure $C$ \citep[p.~52]{GS87}. The marginal processes $\mathcal B_1:[0,1]\to\mathbb R$ and $\mathcal B_2:[0,1]\to\mathbb R$ defined by $\mathcal B_1(u)=\mathcal B(u,1)$ and $\mathcal B_2(u)=\mathcal B(1,u)$ are Brownian bridges. If $R$ is absolutely continuous then it admits an a.e.\ derivative $r\in L^1[0,1]$, and the distributional limit $\mathcal R$ of the P--P process in $L^1[0,1]$ is given by
\begin{equation*}
	\mathcal R(u)=\mathcal B_1(R(u))-r(u)\mathcal B_2(u).
\end{equation*}
The values taken by $r$ at points where $R$ is not differentiable are immaterial; for definiteness one may set $r(u)=1$ on that null set. {Sufficient conditions to obtain a Gaussian limit of this form have previously been obtained for paired observations by \citet{WT21}, while analogous results for independent samples go back considerably further; see, for example, \citet{ACH87} and \citet{HT96}.}

Absolute continuity of $R$ does not imply continuity of $F$ and $G$. {Consider, for instance, the case where $G$ is the discontinuous cdf assigning probability $1/2$ to zero and distributing the remaining probability $1/2$ uniformly on $[1,2]$. The corresponding qf, also discontinuous, is}
\begin{equation*}
	{Q(u)=\begin{cases}0&\text{for }u\in(0,1/2]\\2u&\text{for }u\in(1/2,1).\end{cases}}
\end{equation*}
{If we now set $F=G$ then both cdfs are discontinuous but}
\begin{equation*}
	{R(u)=\begin{cases}1/2&\text{for }u\in[0,1/2]\\u&\text{for }u\in(1/2,1],\end{cases}}
\end{equation*}
{which is absolutely continuous. On the other hand, if we instead choose $F$ to be the cdf of the uniform distribution on $[0,2]$, then $R$ coincides on $(0,1)$ with $(1/2)Q$, which is discontinuous.}

{If $F$ or $G$ is discontinuous then the copula $C$, and consequently the covariance kernel of $\mathcal B$, need not be uniquely determined on all of $[0,1]^2$. This, however, does not lead to any ambiguity in our construction of the distributional limit $\mathcal R$ from $\mathcal B$. The reason is that $\mathcal R$ depends on $\mathcal B$ only through the values taken by the latter process} on $\overline{\ran(F)}\times\overline{\ran(G)}$. Indeed, $\ran(R)\subseteq\overline{\ran(F)}$, and $r(u)=0$ whenever $u\notin\overline{\ran(G)}$. Hence the distribution of $\mathcal R$ is uniquely determined by the construction above, even when $F$ or $G$ is discontinuous.

Let $\ell^\infty[0,1]$ be the space of bounded real-valued functions on $[0,1]$ equipped with the uniform norm. Prior work has identified sufficient conditions for the P--P process to converge in distribution in $\ell^\infty[0,1]$. Absolute continuity of $R$ is not sufficient for such convergence. Indeed, if $r$ is not locally bounded on $(0,1)$ then almost every sample path of $\mathcal R$ is unbounded. To circumvent this issue, prior research relying on convergence in distribution of the P--P process in $\ell^\infty[0,1]$ has sometimes required $R$ to have a bounded derivative on $[0,1]$. See, for instance, \cite{BM15}, \cite{TWT17}, \cite{BS19} and \cite{WT21}. This requirement is convenient, but it excludes some basic examples. For instance, if $F$ and $G$ are normal with different means and equal variances, or with $G$ having a smaller variance than $F$, then $R$ is continuously differentiable on $(0,1)$ but the derivative $r$ diverges to infinity at one or both endpoints. \citet[Theorem 2.2]{HT96} accommodate this feature of the normal case by assuming only that $r$ is locally bounded on $(0,1)$, not necessarily bounded, and asserting convergence in distribution of the P--P process only in the truncated space $\ell^\infty[a,b]$ where $0<a<b<1$.

To obtain convergence in distribution of the P--P process in $\ell^\infty[0,1]$ when $r$ is locally bounded on $(0,1)$ but unbounded near an endpoint, additional control of the endpoint behavior is needed. \citet[Theorem 3.1]{ACH87} impose such control through a \v{C}ibisov--O'Reilly condition. {Specifically, they require that
	\begin{equation*}
		\sup_{u\in(0,1)} r(u)w(u)<\infty
	\end{equation*}
	for some function
	$w:(0,1)\to(0,\infty)$ that is locally bounded
	away from zero, nondecreasing near zero, nonincreasing near one,
	and satisfies}
\begin{equation*}
	{\int_0^1 \frac{1}{u(1-u)}
		\exp\left(-\frac{\delta w(u)^2}{u(1-u)}\right)\,\mathrm du
		<\infty
		\quad\text{for every }\delta>0.}
\end{equation*}
{The subsequent discussion in \citet{ACH87} indicates that this condition on $r$ is close to minimal. It derives from earlier work of \cite{C64} and \cite{O74} on weighted empirical processes.}

Theorem \ref{thm:iff} above contrasts with Theorem 3.1 in \cite{ACH87} because it includes no counterpart to the \v{C}ibisov--O'Reilly condition. The reason is that Theorem \ref{thm:iff} is concerned with convergence in distribution in $L^1[0,1]$, and no assumption on $r$ is needed to ensure that $\mathcal R$ has a.s.\ integrable sample paths. The density $r$ is always integrable because $R$ is nondecreasing and bounded; therefore, since $\mathcal B_1$ and $\mathcal B_2$ have a.s.\ bounded sample paths, it follows that $\mathcal R$ has a.s.\ integrable sample paths. While convergence in distribution of the P--P process in $L^1[0,1]$ is weaker than convergence in distribution in $\ell^\infty[0,1]$, it may be enough to establish good behavior of statistics representable as integral-type functionals of the P--P plot. See \cite{BC26} for an application of Theorem \ref{thm:iff} involving a Wilcoxon--Mann--Whitney statistic.

Theorem \ref{thm:iff} complements the main result of \cite{BK26}, which gives a necessary and sufficient condition for convergence in distribution of the quantile process $\sqrt{n}(Q_n-Q)$ in $L^1(0,1)$. The condition is that $Q$ is locally absolutely continuous and satisfies
\begin{equation*}
	\int_0^1\sqrt{u(1-u)}\,\mathrm{d}Q(u)<\infty.
\end{equation*}
In \cite{BK26}, sufficiency is established by the delta-method. The proof given here for the P--P process also uses the delta-method, but it does not rely on the results of \cite{BK26}. The argument involves two {main} applications of the delta-method: one for the generalized inverse map from cdfs to qfs, and one for the composition of a cdf with a qf. For the generalized inverse map we use the standard Hadamard differentiability result stated in {\citet[p.~532]{VW23}}, which goes back to \cite{V72} and \cite{R76}. Composition requires more work. Standard results on Hadamard differentiability of composition, such as the one stated in \citet[p.~534]{VW23}, assume a uniform differentiability condition stronger than absolute continuity. We therefore prove a tailored lemma in which composition is treated as a map into $L^1[0,1]$. This is Lemma \ref{lem:comp}.

Sections \ref{sec:sufficiency} and \ref{sec:necessity} establish, respectively, the sufficiency and necessity of absolute continuity for convergence in distribution. Section \ref{sec:indep} briefly indicates the modifications needed to handle P--P plots computed from two independent samples of different sizes rather than from a single sample of pairs.

\section{Sufficiency of absolute continuity}\label{sec:sufficiency}

In this section we prove that absolute continuity of $R$ is sufficient for convergence in distribution of the P--P process in $L^1[0,1]$. We also establish a corresponding bootstrap approximation. Both conclusions follow from the same underlying argument based on the delta-method.

Let $W_n=(W_{1,n},\ldots,W_{n,n})$ be a multinomial random vector with {$n$ trials and with} equal probabilities over the categories $1,\ldots,n$, independent of $\{(X_i,Y_i)\}_{i=1}^n$. {We write $\mathbb E_W$ for expectation with respect to the
	bootstrap weights $W_n$, holding the observed sample fixed.} Define the bootstrap empirical cdfs $F^\ast_n:\mathbb R\to[0,1]$ and $G_n^\ast:\mathbb R\to[0,1]$ by
\begin{equation*}
	F_n^\ast(x)=\frac{1}{n}\sum_{i=1}^nW_{i,n}\mathbbm 1(X_i\leq x),\,\,\, G_n^\ast(y)=\frac{1}{n}\sum_{i=1}^nW_{i,n}\mathbbm 1(Y_i\leq y).
\end{equation*}
Let $Q_n^\ast:(0,1)\to\mathbb R$ be the qf corresponding to $G_n^\ast$, and define $R_n^\ast:[0,1]\to[0,1]$ by
\[
R_n^\ast(u) = F_n^\ast(Q_n^\ast(u)) \,\,\, \text{for } u \in (0,1),
\]
\[
R_n^\ast(0)   = \lim_{u\downarrow 0} R_n^\ast(u), \qquad R_n^\ast(1) = \lim_{u\uparrow 1} R_n^\ast(u).
\]

{For a metric space $(\mathbb D,d)$, let $\mathrm{BL}_1(\mathbb D)$ be the collection of all maps $h:\mathbb D\to[-1,1]$ such that $\lvert h(x)-h(y)\rvert\leq d(x,y)$ for all $x,y\in\mathbb D$.}

\begin{theorem}[Sufficiency]\label{thm:sufficiency}
	If $R$ is absolutely continuous then $\sqrt{n}(R_n-R)\rightsquigarrow\mathcal R$ in $L^1[0,1]$, and
	\begin{equation*}
		{\sup_{h\in\mathrm{BL}_1(L^1[0,1])}\big|\mathbb E_Wh(\sqrt{n}(R_n^\ast-R_n))-\mathbb Eh(\mathcal R)\big|\to0}
	\end{equation*}
	{in probability.}
\end{theorem}

We prove this theorem by combining Donsker's theorem with the delta-method and its {conditional analogue}. The main technical point is to verify Hadamard differentiability of a composition map under absolute continuity alone. Standard composition results, such as Lemma 3.10.28 in \cite{VW23}, are not directly applicable here because they require a stronger uniform differentiability condition. We therefore begin with a tailored lemma for composition as a map into $L^1[0,1]$.

For two normed spaces $\mathcal X$ and $\mathcal Y$ we denote by $\mathcal X\otimes\mathcal Y$ the product space equipped with the max-norm. Define subsets $\mathbb D_1$ and $\mathbb D_2$ of $\ell^\infty[0,1]$ by
\begin{align*}
	\mathbb D_1&=\{A\in\ell^\infty[0,1]:\ran(A)\subseteq[0,1]\\
	&\quad\quad\text{ and }A\text{ is Lebesgue measurable}\}\\
	\text{and}\quad\mathbb D_2&=\{B\in\ell^\infty[0,1]:B\text{ is Borel measurable}\},
\end{align*}
{where we recall that a map $h:[0,1]\to\mathbb R$ is called Lebesgue measurable if $h^{-1}(E)$ is a Lebesgue measurable subset of $[0,1]$ for every Borel subset $E$ of $\mathbb R$.} Define the composition map $\phi:\mathbb D_1\times\mathbb D_2\to L^1[0,1]$ by
\begin{equation*}
	\phi(A,B)(u)=B\circ A(u)=B(A(u)),\quad u\in[0,1].
\end{equation*}
The Lebesgue measurability of $A$ and Borel measurability of $B$ together guarantee that $B\circ A$ is Lebesgue measurable, and thus the boundedness of $B$ guarantees that $B\circ A\in L^1[0,1]$. {Note that if both $A$ and $B$ were merely required to be Lebesgue measurable then it would not necessarily be true that $B\circ A$ is Lebesgue measurable; see \citet[p.~83]{H74}.}

Let $C[0,1]$ be the space of continuous real-valued functions on $[0,1]$ equipped with the uniform norm, a subspace of $\ell^\infty[0,1]$. Let $I:[0,1]\to[0,1]$ be the identity map.

\begin{lemma}[Hadamard differentiability of composition]\label{lem:comp}
	Let $B:[0,1]\to\mathbb R$ be absolutely continuous with density $b:[0,1]\to\mathbb R$. Then the composition map $\phi:\mathbb D_1\times\mathbb D_2\subset\ell^\infty[0,1]\otimes\ell^\infty[0,1]\to L^1[0,1]$ is Hadamard differentiable at $(I,B)$ tangentially to $C[0,1]\times C[0,1]$. The derivative $\phi'_{I,B}:C[0,1]\times C[0,1]\to L^1[0,1]$ is given by $\phi'_{I,B}(\alpha,\beta)=\beta+b\alpha$.
\end{lemma}
\begin{proof}
	Let $\{t_n\}_{n=1}^\infty$ be a sequence of positive real numbers such that $t_n\to0$. Let $\{\alpha_n\}_{n=1}^\infty$ be a sequence in $\ell^\infty[0,1]$ such that $I+t_n\alpha_n\in\mathbb D_1$ for all $n$ and such that $\alpha_n\to\alpha\in C[0,1]$. Let $\{\beta_n\}_{n=1}^\infty$ be a sequence in $\ell^\infty[0,1]$ such that $B+t_n\beta_n\in\mathbb D_2$ for all $n$ and such that $\beta_n\to\beta\in C[0,1]$. It suffices to show that $\lVert t_n^{-1}[(B+t_n\beta_n)\circ(I+t_n\alpha_n)-B]-\beta-b\alpha\rVert_1\to0$. By the triangle inequality, this norm is bounded by the following sum of three terms:
	\begin{multline*}
		\lVert(\beta_n-\beta)\circ(I+t_n\alpha_n)\rVert_1+\lVert\beta\circ(I+t_n\alpha_n)-\beta\rVert_1\\+\lVert t_n^{-1}[B\circ(I+t_n\alpha_n)-B]-b\alpha\rVert_1.
	\end{multline*}
	The first term converges to zero because $\beta_n$ converges uniformly to $\beta$.  The second term converges to zero by the dominated convergence theorem because $\beta$ is continuous. It remains to show that the third term converges to zero.
	
	Since $B$ is absolutely continuous with density $b$ we have, for each $u\in[0,1]$,
	\begin{multline*}
		t_n^{-1}[B(u+t_n\alpha_n(u))-B(u)]\\=t_n^{-1}\int_0^1[\mathbbm 1(v\leq u+t_n\alpha_n(u))-\mathbbm 1(v\leq u)]b(v)\,\mathrm{d}v.
	\end{multline*}
	Define the map $\zeta_n\in\mathbb D_2$ by
	\begin{equation*}
		\zeta_n(u)=t_n^{-1}\int_0^1[\mathbbm 1(v\leq u+t_n\alpha(u))-\mathbbm 1(v\leq u)]b(v)\,\mathrm{d}v.
	\end{equation*}
	Then by Fubini's theorem we have
	\begin{multline*}
		\lVert t_n^{-1}[B\circ(I+t_n\alpha_n)-B]-\zeta_n\rVert_1\\\leq t_n^{-1}\!\!\int_0^1\!\!\!\int_0^1\! \lvert\mathbbm 1(v\leq u+t_n\alpha_n(u))-\mathbbm 1(v\leq u+t_n\alpha(u))\rvert\lvert b(v)\rvert\mathrm{d}u\mathrm{d}v.
	\end{multline*}
	Using the change-of-variables $u\mapsto v-t_nw$ we may rewrite the right-hand side of the inequality as
	\begin{multline*}
		\int_0^1\!\int_{-\infty}^\infty\mathbbm 1(t_n^{-1}(v-1)\leq w\leq t_n^{-1}v)\,\bigl|\mathbbm 1(w\leq\alpha_n(v-t_nw))-\\\mathbbm 1(w\leq\alpha(v-t_nw))\bigr|\,\lvert b(v)\rvert\,\mathrm{d}w\,\mathrm{d}v.
	\end{multline*}
	The integrand, as a function of $v$ and $w$, is bounded by $\mathbbm 1(\lvert w\rvert\leq\lVert\alpha\rVert_\infty+1)\lvert b(v)\rvert$ for all sufficiently large $n$, and converges a.e.\ to zero because $\alpha_n$ converges uniformly to $\alpha$ and because $\alpha$ is continuous. (Convergence of the integrand to zero need not hold on the set $\{(v,w):w=\alpha(v)\}$, but this is a null set.) Thus the double integral over $w$ and $v$ converges to zero by the dominated convergence theorem. This shows that
	\begin{equation*}
		\lVert t_n^{-1}[B\circ(I+t_n\alpha_n)-B]-\zeta_n\rVert_1\to0.
	\end{equation*}	
	It therefore remains to show that $\lVert\zeta_n-b\alpha\rVert_1\to0$.
	
	Define the map $\chi_n:[0,1]^2\to\mathbb R$ by
	\begin{equation*}
		\chi_n(u,v)=t_n^{-1}[\mathbbm 1(v\leq u+t_n\alpha(u))-\mathbbm 1(v\leq u)],
	\end{equation*}	
	so that
	\begin{equation*}
		\zeta_n(u)=\int_0^1\chi_n(u,v)b(v)\,\mathrm{d}v
	\end{equation*}
	for each $u\in[0,1]$. Then
	\begin{align}
		\lVert\zeta_n-b\alpha\rVert_1&\leq \int_0^1\left\vert b(u)\int_0^1\chi_n(u,v)\,\mathrm{d}v-b(u)\alpha(u)\right\vert\mathrm{d}u\notag\\
		&\quad\quad+\int_0^1\left\vert\int_0^1\chi_n(u,v)(b(v)-b(u))\,\mathrm{d}v\right\vert\mathrm{d}u.\label{eq:comp2terms}
	\end{align}
	We now show that the first term on the right-hand side converges to zero. Use the change-of-variables $v\mapsto u+t_nw$ to obtain, for each $u\in[0,1]$,
	\begin{multline}
		\int_0^1\chi_n(u,v)\,\mathrm{d}v=\int_{-u/t_n}^{(1-u)/t_n}\chi_n(u,u+t_nw)\,t_n\,\mathrm{d}w\\
		=\int_{-u/t_n}^{(1-u)/t_n}[\mathbbm 1(w\leq\alpha(u))-\mathbbm 1(w\leq0)]\,\mathrm{d}w.\label{eq:compterm1}
	\end{multline}
	The final integral converges to $\alpha(u)$ for each $u\in(0,1)$ and is bounded by $\lVert\alpha\rVert_\infty$. Consequently
	\begin{equation*}
		b(u)\int_0^1\chi_n(u,v)\,\mathrm{d}v-b(u)\alpha(u)
	\end{equation*}	
	converges to zero for each $u\in(0,1)$ and is bounded by $2\lVert\alpha\rVert_\infty\lvert b(u)\rvert$, an integrable function of $u$. Thus an application of the dominated convergence theorem shows that the first term on the right-hand side of \eqref{eq:comp2terms} converges to zero. It remains to show that the second term converges to zero.
	
	Fix $\epsilon>0$. The continuous maps from $[0,1]$ to $\mathbb R$ are dense in $L^1[0,1]$, so there exists a continuous $b_\epsilon:[0,1]\to\mathbb R$ such that $\lVert b_\epsilon-b\rVert_1\leq\epsilon$. Use Fubini's theorem to bound the second term on the right-hand side of \eqref{eq:comp2terms} by
	\begin{multline}
		\int_0^1\left(\int_0^1\lvert\chi_n(u,v)\rvert\,\mathrm{d}v\right)\lvert b_\epsilon(u)-b(u)\rvert\,\mathrm{d}u\\+\int_0^1\left(\int_0^1\lvert\chi_n(u,v)\rvert\,\mathrm{d}u\right)\lvert b_\epsilon(v)-b(v)\rvert\,\mathrm{d}v\\+\int_0^1\int_0^1\lvert\chi_n(u,v)\rvert\,\lvert b_\epsilon(u)-b_\epsilon(v)\rvert\,\mathrm{d}v\,\mathrm{d}u.\label{eq:comptriangle}
	\end{multline}
	The first term is bounded by $\lVert\alpha\rVert_\infty\epsilon$ because we have $\smallint_0^1\lvert\chi_n(u,v)\rvert\,\mathrm{d}v\leq\lVert\alpha\rVert_\infty$ by applying the change-of-variables $v\mapsto u+t_nw$ as in \eqref{eq:compterm1}. To bound the second term we use the change-of-variables $u\mapsto v-t_nw$ to obtain
	\begin{align*}
		\int_0^1\lvert\chi_n(u,v)\rvert\,\mathrm{d}u&=\int_{(v-1)/t_n}^{v/t_n}\lvert\chi_n(v-t_nw,v)\rvert\,t_n\,\mathrm{d}w\\&{=}\int_{{(v-1)/t_n}}^{v/t_n}\!\lvert\mathbbm 1(w\leq\alpha(v-t_nw))-\mathbbm 1(w\leq0)\rvert\,\mathrm{d}w.
	\end{align*}
	The final integrand can be nonzero only if $\lvert w\rvert\leq\lVert\alpha\rVert_\infty$. Since $\alpha$ is uniformly continuous there exists an integer $N_{\epsilon,1}$ such that $\lvert\alpha(v-t_nw)-\alpha(v)\rvert\leq\epsilon$ for all $n\geq N_{\epsilon,1}$, all $v\in[0,1]$ and all $w\in{[(v-1)/t_n,v/t_n]}$ such that $\lvert w\rvert\leq\lVert\alpha\rVert_\infty$. Consequently
	\begin{align*}
		&\int_{{(v-1)/t_n}}^{v/t_n}\lvert\mathbbm 1(w\leq\alpha(v-t_nw))-\mathbbm 1(w\leq0)\rvert\,\mathrm{d}w\\&\quad\leq\int_{-\infty}^\infty\mathbbm 1\big([0\wedge(\alpha(v)-\epsilon)]\leq w\leq[0\vee(\alpha(v)+\epsilon)]\big)\,\mathrm{d}w\\&\quad\leq\lvert\alpha(v)\rvert+2\epsilon\leq\lVert\alpha\rVert_\infty+2\epsilon
	\end{align*}
	for all $n\geq N_{\epsilon,1}$ and all $v\in[0,1]$. The second term in \eqref{eq:comptriangle} is therefore bounded by $\lVert\alpha\rVert_\infty\epsilon+2\epsilon^2$ for all $n\geq N_{\epsilon,1}$. To bound the third term, observe that since $b_\epsilon$ is uniformly continuous there exists a real number $\delta_{\epsilon}>0$ such that $\lvert b_\epsilon(u)-b_\epsilon(v)\rvert\leq\epsilon$ whenever $\lvert u-v\rvert\leq\delta_{\epsilon}$. Further observe that $\chi_n(u,v)$ can be nonzero only if $\lvert u-v\rvert\leq t_n\lVert\alpha\rVert_\infty$. Therefore there exists an integer $N_{\epsilon,2}$ such that, for all $n\geq N_{\epsilon,2}$, the third term in \eqref{eq:comptriangle} is bounded by $\epsilon\smallint_0^1\smallint_0^1\lvert\chi_n(u,v)\rvert\,\mathrm{d}v\,\mathrm{d}u$. Thus, since $\smallint_0^1\lvert\chi_n(u,v)\rvert\,\mathrm{d}v\leq\lVert\alpha\rVert_\infty$, the third term is bounded by $\lVert\alpha\rVert_\infty\epsilon$ for all $n\geq N_{\epsilon,2}$.
	
	Combining the preceding bounds shows that the second term on the right-hand side of \eqref{eq:comp2terms} is at most $3\lVert\alpha\rVert_\infty\epsilon+2\epsilon^2$ for all sufficiently large $n$. Therefore, since $\epsilon$ may be chosen arbitrarily small, the second term converges to zero.
\end{proof}

Lemma \ref{lem:comp} supplies the only nonstandard differentiability argument needed for the proof of Theorem \ref{thm:sufficiency}. The remainder of the proof combines Donsker's theorem with the delta-method and its {conditional} analogue.

\begin{proof}[Proof of Theorem \ref{thm:sufficiency}]
	Let $\{(U_i,V_i)\}_{i=1}^\infty$ be an iid sequence of pairs of uniform random variables such that the bivariate distribution of each pair $(U_i,V_i)$ is given by the copula $C$. Let $C_n:[0,1]^2\to[0,1]$ be the bivariate empirical cdf for $\{(U_i,V_i)\}_{i=1}^n$. Let $\tilde{F}_n:[0,1]\to[0,1]$ and $\tilde{G}_n:[0,1]\to[0,1]$ be the empirical cdfs for $\{U_i\}_{i=1}^n$ and $\{V_i\}_{i=1}^n$, respectively. Let $\tilde{Q}_n:(0,1)\to[0,1]$ be the empirical qf for $\{V_i\}_{i=1}^n$. {Donsker's theorem, applied using the VC-class of lower-left rectangles $\{[0,u]\times[0,v]:(u,v)\in[0,1]^2\}$---see Theorem 2.5.2 and Example 2.5.4 in \citet{VW23}---}establishes that
	\begin{equation*}
		\sqrt{n}(C_n-C)\rightsquigarrow\mathcal B\,\,\,\text{in }\ell^\infty([0,1]^2).
	\end{equation*}	
	{Now an application of the delta method---see Theorem 3.10.4 in \citet{VW23}---using the linear map $h(u,v)\mapsto(h(1,v),h(R(u),1))$ from $\ell^\infty([0,1]^2)$ into $\ell^\infty[0,1]\otimes\ell^\infty[0,1]$} shows that
	\begin{equation}\label{eq:donskerjoint}
		\sqrt{n}((\tilde{G}_n,\tilde{F}_n\circ R)-(I,R))\rightsquigarrow(\mathcal B_2,\mathcal B_1\circ R)\,\,\,\text{in }\ell^\infty[0,1]\otimes\ell^\infty[0,1],
	\end{equation}
	where $\mathcal B_1(\cdot)=\mathcal B(\cdot,1)$ and $\mathcal B_2(\cdot)=\mathcal B(1,\cdot)$. {A second} application of the delta-method using the generalized inverse map, suitable Hadamard differentiability of which is supplied by {Lemma 3.10.24 in \citet{VW23}}, shows that
	\begin{equation*}
		\sqrt{n}((\tilde{Q}_n,\tilde{F}_n\circ R)-(I,R))\rightsquigarrow(-\mathcal B_2,\mathcal B_1\circ R)\,\,\,\text{in }\ell^\infty[0,1]\otimes\ell^\infty[0,1].
	\end{equation*}
	{A third application of the delta-method using the composition map $\phi$ defined above, suitable Hadamard differentiability of which is supplied by Lemma \ref{lem:comp}, yields}
	\begin{equation}\label{eq:conv1}
		\sqrt{n}(\tilde{F}_n\circ R\circ\tilde{Q}_n-R)\rightsquigarrow\mathcal B_1\circ R-r\mathcal B_2=\mathcal R\quad\text{in }L^1[0,1].
	\end{equation}
	
	It remains to replace $\tilde{F}_n\circ R\circ\tilde{Q}_n$ with $R_n$ in the preceding display. Let $P:(0,1)\to\mathbb R$ be the qf for $X_i$. The function $\tilde{F}_n\circ F$ is the empirical cdf for $\{P(U_i)\}_{i=1}^n$ because
	\begin{equation*}
		\tilde{F}_n(F(x))=\frac{1}{n}\sum_{i=1}^n\mathbbm 1(U_i\leq F(x))=\frac{1}{n}\sum_{i=1}^n\mathbbm 1(P(U_i)\leq x).
	\end{equation*}
	The function $Q\circ \tilde{Q}_n$ is the empirical qf for $\{Q(V_i)\}_{i=1}^n$ because applying $Q$ to the order statistics of $\{V_i\}_{i=1}^n$ produces the order statistics of $\{Q(V_i)\}_{i=1}^n$. We have
	\begin{multline}\label{eq:eqd}
		\big((P(U_1),Q(V_1)),\ldots,(P(U_n),Q(V_n))\big)\\\eqd\big((X_1,Y_1),\ldots,(X_n,Y_n)\big)
	\end{multline}
	because the distribution of each pair $(P(U_i),Q(V_i))$ has margins $F$ and $G$ and admits the copula $C$. Therefore
	\begin{equation*}
		(\tilde{F}_n\circ F)\circ(Q\circ\tilde{Q}_n)\eqd F_n\circ Q_n
	\end{equation*}
	as random elements of $L^1[0,1]$. Since
	\begin{equation*}
		(\tilde{F}_n\circ F)\circ(Q\circ\tilde{Q}_n)=\tilde{F}_n\circ R\circ\tilde{Q}_n\quad\text{and}\quad F_n\circ Q_n=R_n
	\end{equation*}
	on $(0,1)$, we deduce that $\tilde{F}_n\circ R\circ\tilde{Q}_n\eqd R_n$ as random elements of $L^1[0,1]$, which justifies replacing $\tilde{F}_n\circ R\circ\tilde{Q}_n$ with $R_n$ in \eqref{eq:conv1}. This proves the first assertion of the theorem.
	
	The proof of the second assertion is similar. Define $C_n^\ast:[0,1]^2\to[0,1]$, $\tilde{F}_n^\ast:[0,1]\to[0,1]$ and $\tilde{G}_n^\ast:[0,1]\to[0,1]$ by
	\[C_n^\ast(u,v)=\frac{1}{n}\sum_{i=1}^nW_{i,n}\mathbbm 1(U_i\leq u,V_i\leq v),\]
	\[\tilde{F}_n^\ast(u)=C_n^\ast(u,1),\quad\tilde{G}_n^\ast(v)=C_n^\ast(1,v).\]
	Let $\tilde{Q}_n^\ast$ be the qf corresponding to $\tilde{G}_n^\ast$. The bootstrap version of Donsker's theorem{---see Theorem 3.7.1 in \citet{VW23}---establishes that $\sqrt{n}(C_n^\ast-C_n)$ is asymptotically measurable as a map into $\ell^\infty([0,1]^2)$ and}
	\begin{equation*}
		{\sup_{h\in\mathrm{BL}_1(\ell^\infty([0,1]^2))}\big|\mathbb E_Wh(\sqrt{n}(C_n^\ast-C_n))-\mathbb Eh(\mathcal B)\big|\to0}
	\end{equation*}
	{in outer probability. Now we apply the delta method three times as above, but this time using the conditional version of the delta method provided by Theorem 3.10.11 in \citet{VW23}, to deduce that}
	\begin{equation*}
		{\sup_{h\in\mathrm{BL}_1(L^1[0,1])}\big|\mathbb E_Wh(\sqrt{n}(\tilde{F}_n^\ast\circ R\circ\tilde{Q}_n^\ast-\tilde{F}_n\circ R\circ\tilde{Q}_n)-\mathbb Eh(\mathcal R)\big|}
	\end{equation*}
	{converges to zero in outer probability. Such convergence is, by Theorem 1.13.1 in \citet{VW23}, equivalent to having
		\begin{equation}\label{eq:conv2}
			\mathbb E_Wh(\sqrt{n}(\tilde{F}_n^\ast\circ R\circ\tilde{Q}_n^\ast-\tilde{F}_n\circ R\circ\tilde{Q}_n)\to\mathbb Eh(\mathcal R)
		\end{equation}
		in probability for every bounded Lipschitz $h:L^1[0,1]\to\mathbb R$.
		
		Fix a bounded Lipschitz map $h:L^1[0,1]\to\mathbb R$. For
		\begin{equation*}
			z=((x_1,y_1),\ldots,(x_n,y_n))\in(\mathbb R^2)^n,
		\end{equation*}
		let $T_n(z)\in L^1[0,1]$ denote the P--P plot based on the sample $z$, and, for a bootstrap weight vector $w\in\mathbb R^n$, let $T_n^\ast(z,w)\in L^1[0,1]$ denote the corresponding bootstrap P--P plot. Define the map $\Gamma_{n,h}:(\mathbb R^2)^n\to\mathbb R$ by
		\begin{equation*}
			\Gamma_{n,h}(z)=\mathbb E_Wh\big(\!\sqrt n(T_n^\ast(z,W_n)-T_n(z))\big).
		\end{equation*}
		Note that $\Gamma_{n,h}$ is Borel measurable because, for each fixed $w$ in the finite support of $W_n$, the maps $z\mapsto T_n(z)$ and $z\mapsto T_n^\ast(z,w)$ from $(\mathbb R^2)^n$ into $L^1[0,1]$ are Borel measurable. To see this directly, note that each map has finite range and that the preimage of each possible P--P plot is determined by finitely many inequalities among the sample coordinates.
		
		We showed earlier that $\tilde{F}_n\circ R\circ\tilde{Q}_n$ is the P--P plot for the transformed sample $\{(P(U_i),Q(V_i))\}_{i=1}^n$. The same argument shows that $\tilde{F}^\ast_n\circ R\circ\tilde{Q}^\ast_n$ is the corresponding bootstrap P--P plot. Consequently
		\begin{multline*}
			\Gamma_{n,h}\big((P(U_1),Q(V_1)),\ldots,(P(U_n),Q(V_n))\big)\\=\mathbb E_Wh\big(\sqrt{n}(\tilde{F}_n^\ast\circ R\circ\tilde{Q}_n^\ast-\tilde{F}_n\circ R\circ\tilde{Q}_n)\big).
		\end{multline*}
		It is also the case that
		\begin{equation*}
			\Gamma_{n,h}\big((X_1,Y_1),\ldots,(X_n,Y_n)\big)=\mathbb E_Wh\big(\sqrt{n}(R_n^\ast-R_n)\big).
		\end{equation*}
		Thus we deduce from \eqref{eq:eqd} and from the fact that $\Gamma_{n,h}$ is Borel measurable that
		\begin{equation*}
			\mathbb E_Wh\big(\!\sqrt{n}(R_n^\ast-R_n)\big)\eqd\mathbb E_Wh\big(\!\sqrt{n}(\tilde{F}_n^\ast\circ R\circ\tilde{Q}_n^\ast-\tilde{F}_n\circ R\circ\tilde{Q}_n)\big),
		\end{equation*}
		and so deduce from \eqref{eq:conv2} that
		\begin{equation*}
			\mathbb E_Wh\big(\!\sqrt{n}(R_n^\ast-R_n)\big)\to\mathbb Eh(\mathcal R)\quad\text{in probability.}
		\end{equation*}
		Since this is true for every bounded Lipschitz $h:L^1[0,1]\to\mathbb R$, another application of Theorem 1.13.1 in \citet{VW23} establishes that
		\begin{equation*}
			\sup_{h\in\mathrm{BL}_1(L^1[0,1])}\big|\mathbb E_Wh(\sqrt{n}(R_n^\ast-R_n))-\mathbb Eh(\mathcal R)\big|\to0
		\end{equation*}
		in outer probability. In fact, we have convergence to zero in probability, not merely in outer probability, because the supremum over $\mathrm{BL}_1(L^1[0,1])$ is measurable due to the separability of $L^1[0,1]$; see Definition 2.2 and the subsequent discussion in \citet{BZ16}. This proves the second assertion.}
\end{proof}

\section{Necessity of absolute continuity}\label{sec:necessity}

We now show that convergence in distribution in $L^1[0,1]$ cannot hold unless $R$ is absolutely continuous. For clarity we state this as a separate theorem.

\begin{theorem}[Necessity]\label{thm:necessity}
	If $R$ is not absolutely continuous then $\sqrt{n}(R_n-R)$ does not converge in distribution in $L^1[0,1]$.
\end{theorem}

The proof is similar in spirit to the necessity argument in \cite{BK26}, but the P--P setting requires a separate treatment because the process involves composition of the empirical cdf with the empirical quantile function. We begin with two lemmas. Let $\bar R:(0,1)\to[0,1]$ be the left-continuous version of $R$ defined by $\bar R(u)=\sup_{v\in(0,u)}R(v)$. This construction makes $\bar R$ a valid qf. In fact $\bar R$ is the qf for each $F(Y_i)$.
\begin{lemma}\label{lem:Rbar}
	$\bar{R}$ is the qf for each $F(Y_i)$ and $F\circ Q_n$ is the empirical qf for $\{F(Y_i)\}_{i=1}^n$.
\end{lemma}
\begin{proof}
	Let $U$ be a random variable distributed uniformly on $(0,1)$. Then $\bar{R}(U)\eqas R(U)=F(Q(U))\eqd F(Y_i)$. This shows that $\bar{R}$ is the qf for each $F(Y_i)$. The second assertion of the lemma is true because applying $F$ to the order statistics of $\{Y_i\}_{i=1}^n$ produces the order statistics of $\{F(Y_i)\}_{i=1}^n$.
\end{proof}
\begin{lemma}\label{lem:ui}
	$\{\!\sqrt{n}\lVert R_n-R\rVert_1\}_{n=1}^\infty$ is uniformly integrable.
\end{lemma}
\begin{proof}
	Let $H:[0,1]\to[0,1]$ be the cdf for each $F(Y_i)$ and let $H_n:[0,1]\to[0,1]$ be the empirical cdf for $\{F(Y_i)\}_{i=1}^n$. The $L^1$-distance between any two cdfs is equal to the $L^1$-distance between the corresponding two qfs. It therefore follows from Lemma \ref{lem:Rbar} that $\lVert H_n-H\rVert_1=\lVert F\circ Q_n-\bar{R}\rVert_1$. Observe that
	\begin{align*}
		\sqrt{n}\rVert R_n-R\rVert_1&\leq\sqrt{n}\lVert F_n\circ Q_n-F\circ Q_n\rVert_1+\sqrt{n}\lVert F\circ Q_n-\bar{R}\rVert_1\\&\quad\quad+\sqrt{n}\lVert \bar{R}-R\rVert_1.
	\end{align*}
	The first term is bounded by $\sqrt{n}\lVert F_n-F\rVert_\infty$, the second term is equal to $\sqrt{n}\lVert H_n-H\rVert_1$ and thus bounded by $\sqrt{n}\lVert H_n-H\rVert_\infty$, and the third term is zero. Therefore it suffices that
	\begin{equation*}
		\{\!\sqrt{n}\lVert F_n-F\rVert_\infty\}_{n=1}^\infty\quad\text{ and }\quad\{\!\sqrt{n}\lVert H_n-H\rVert_\infty\}_{n=1}^\infty
	\end{equation*}
	are uniformly integrable. This is a well-known consequence of the Dvoretzky--Kiefer--Wolfowitz inequality {\citep[p.~346]{VW23}.}
\end{proof}
\begin{proof}[Proof of Theorem \ref{thm:necessity}]
	We prove the contrapositive. Let $\mathcal S$ be a random element of $L^1[0,1]$ and assume that $\sqrt{n}(R_n-R)\rightsquigarrow\mathcal S$ in $L^1[0,1]$. Let $\mathcal B_2:[0,1]\to\mathbb R$ be a Brownian bridge. We will establish the bound
	\begin{multline}\label{eq:openineqR}
		\mathbb E\int_{[a,b)}\lvert \mathcal B_2(u)\rvert\,\mathrm{d}R(u)\leq\mathbb E\int_a^b\lvert\mathcal S(u)\rvert\,\mathrm{d}u+\sqrt{\frac{\pi}{2}}(b-a)\\\text{for all continuity points }a,b\text{ of }R\text{ with }0<a<b<1.
	\end{multline}
	The first integral is defined in the Lebesgue--Stieltjes sense, wherein a Borel measure $\mu_R$ on $(0,1)$ is generated by assigning measure $\bar{R}(b)-\bar{R}(a)$ to each interval $[a,b)$. Observe that
	\begin{multline*}
		\mathbb E\int_a^b\lvert F(Q_n(u))-R(u)\rvert\,\mathrm{d}u\leq\mathbb E\int_a^b\lvert R_n(u)-R(u)\rvert\,\mathrm{d}u\\+\mathbb E\int_a^b\lvert F_n(Q_n(u))-F(Q_n(u))\rvert\,\mathrm{d}u.
	\end{multline*}
	Consequently, to establish the inequality in \eqref{eq:openineqR} it suffices to verify that
	\begin{equation}
		\lim_{n\to\infty}\mathbb E\sqrt{n}\int_a^b\lvert R_n(u)-R(u)\rvert\,\mathrm{d}u=\mathbb E\int_a^b\lvert\mathcal S(u)\rvert\,\mathrm{d}u,\label{eq:Rverify1}
	\end{equation}
	that
	\begin{equation}
		\limsup_{n\to\infty}\mathbb E\sqrt{n}\int_a^b\lvert F_n(Q_n(u))-F(Q_n(u))\rvert\,\mathrm{d}u\leq\sqrt{\frac{\pi}{2}}(b-a),\label{eq:Rverify2}
	\end{equation}
	and that
	\begin{multline}
		\liminf_{n\to\infty}\mathbb E\sqrt{n}\int_a^b\lvert F(Q_n(u))-R(u)\rvert\,\mathrm{d}u\\\geq\mathbb E\int_{[a,b)}\lvert \mathcal B_2(u)\rvert\,\mathrm{d}R(u).\label{eq:Rverify3}
	\end{multline}
	An application of the continuous mapping theorem shows that $\sqrt{n}\smallint_a^b\lvert R_n(u)-R(u)\rvert\,\mathrm{d}u\rightsquigarrow\smallint_a^b\lvert\mathcal S(u)\rvert\,\mathrm{d}u$. Convergence in distribution implies convergence in mean due to the uniform integrability established in Lemma \ref{lem:ui}. This verifies \eqref{eq:Rverify1} and also shows that $\mathbb E\smallint_0^1\lvert\mathcal S(u)\rvert\,\mathrm{d}u<\infty$. Observe that
	\begin{equation*}
		\int_a^b\lvert F_n(Q_n(u))-F(Q_n(u))\rvert\,\mathrm{d}u\leq\lVert F_n-F\rVert_\infty(b-a).
	\end{equation*}	
	We have $\mathbb E\sqrt{n}\lVert F_n-F\rVert_\infty\leq\sqrt{\pi/2}$ by the Dvoretzky--Kiefer--Wolfowitz inequality. This verifies \eqref{eq:Rverify2}. Lemma 3.4 in \cite{BK26}, applied to the sample $\{F(Y_i)\}_{i=1}^n$ and combined with Lemma \ref{lem:Rbar} above, yields \eqref{eq:Rverify3} whenever $a$ and $b$ are continuity points of $R$.
	
	Next we strengthen \eqref{eq:openineqR} by showing that
	\begin{multline}\label{eq:openineqR2}
		\mathbb E\int_{[a,b]}\lvert \mathcal B_2(u)\rvert\,\mathrm{d}R(u)\leq\mathbb E\int_a^b\lvert\mathcal S(u)\rvert\,\mathrm{d}u+\sqrt{\frac{\pi}{2}}(b-a)\\\text{for all }a,b\in(0,1)\text{ such that }a<b.
	\end{multline}
	Let $\{a_{k}\}_{{k}=1}^\infty$ and $\{b_{k}\}_{{k}=1}^\infty$ be sequences of continuity points of $R$ such that $0<a_{k}<b_{k}<1$ for each ${k}$ and such that $a_{k}\uparrow a$ and $b_{k}\downarrow b$. Then Fatou's lemma and \eqref{eq:openineqR} respectively justify the inequalities
	\begin{multline*}
		\mathbb E\int_{[a,b]}\lvert \mathcal B_2(u)\rvert\,\mathrm{d}R(u)\leq\liminf_{{k}\to\infty}\mathbb E\int_{[a_{k},b_{k})}\lvert \mathcal B_2(u)\rvert\,\mathrm{d}R(u)\\\leq\liminf_{{k}\to\infty}\left(\mathbb E\int_{a_{k}}^{b_{k}}\lvert\mathcal S(u)\rvert\,\mathrm{d}u+\sqrt{\frac{\pi}{2}}(b_{k}-a_{k})\right).
	\end{multline*}
	Since $\mathbb E\smallint_0^1\lvert\mathcal S(u)\rvert\,\mathrm{d}u<\infty$, an application of the dominated convergence theorem yields \eqref{eq:openineqR2}.
	
	Finally we use \eqref{eq:openineqR2} to show that $R$ is absolutely continuous on $[0,1]$. It suffices to show absolute continuity on $(0,1)$ because $R$ is continuous at zero and one by construction. Let $\lambda$ be the Lebesgue measure on $(0,1)$, let $\mu_R$ be the Lebesgue--Stieltjes measure on $(0,1)$ generated by $R$, and let $\mathcal J$ be the collection of all finite unions of closed intervals with endpoints in $(0,1)$. By definition, $R$ is absolutely continuous on $(0,1)$ if for every $\epsilon>0$ there exists $\delta_\epsilon>0$ such that $\mu_R(A)\leq\epsilon$ for every $A\in\mathcal J$ such that $\lambda(A)\leq\delta_\epsilon$. For every $A\in\mathcal J$ and every $\eta\in(0,1/2)$ we have			
	\begin{align*}
		\mu_R(A)&=\int_A\mathbbm 1(\eta\leq u\leq1-\eta)\,\mathrm{d}R(u)+\int_A\mathbbm 1(u<\eta)\,\mathrm{d}R(u)\\
		&\quad\quad+\int_A\mathbbm 1(u>1-\eta)\,\mathrm{d}R(u)\\
		&\leq\frac{1}{\sqrt{\eta(1-\eta)}}\int_A\sqrt{u(1-u)}\,\mathrm{d}R(u)\\
		&\quad\quad+\left(\bar{R}(\eta)-\lim_{u\downarrow0}\bar{R}(u)\right)+\left(\lim_{u\uparrow1}\bar{R}(u)-\bar{R}(1-\eta)\right).
	\end{align*}
	Fix $\epsilon>0$. Since $\lim_{u\downarrow0}\bar{R}(u)$ and $\lim_{u\uparrow1}\bar{R}(u)$ are finite there exists $\eta_\epsilon\in(0,1/2)$ such that, for every $A\in\mathcal J$,
	\begin{equation*}
		\mu_R(A)\leq\frac{1}{\sqrt{\eta_\epsilon(1-\eta_\epsilon)}}\int_A\sqrt{u(1-u)}\,\mathrm{d}R(u)+\frac{\epsilon}{3}.
	\end{equation*}
	We have $\mathbb E\lvert\mathcal B_2(u)\rvert=\sqrt{(2/\pi)u(1-u)}$ because $\mathcal B_2(u)$ is normally distributed with mean zero and variance $u(1-u)$. Therefore an application of Fubini's theorem shows that
	\begin{equation*}
		\mu_R(A)\leq\sqrt{\frac{\pi}{2\eta_\epsilon(1-\eta_\epsilon)}}\,\mathbb E\int_A\lvert\mathcal B_2(u)\rvert\,\mathrm{d}R(u)+\frac{\epsilon}{3}.
	\end{equation*}
	Since $A$ is a finite union of closed intervals with endpoints in $(0,1)$, an application of the bound in \eqref{eq:openineqR2} shows that
	\begin{equation*}
		\mu_R(A)\leq\sqrt{\frac{\pi}{2\eta_\epsilon(1-\eta_\epsilon)}}\left(\mathbb E\int_A\lvert\mathcal S(u)\rvert\,\mathrm{d}u+\sqrt{\frac{\pi}{2}}\lambda(A)\right)+\frac{\epsilon}{3}.
	\end{equation*}
	For each ${N}\in\mathbb N$ we have
	\begin{align*}
		\mathbb E\int_A\lvert\mathcal S(u)\rvert\,\mathrm{d}u&=\mathbb E\int_A\mathbbm 1\big(\lvert\mathcal S(u)\rvert\leq {N}\big)\lvert\mathcal S(u)\rvert\,\mathrm{d}u\\
		&\quad\quad+\mathbb E\int_A\mathbbm 1\big(\lvert\mathcal S(u)\rvert>{N}\big)\lvert\mathcal S(u)\rvert\,\mathrm{d}u\\
		&\leq {N}\lambda(A)+\mathbb E\int_0^1\mathbbm 1\big(\lvert\mathcal S(u)\rvert>{N}\big)\lvert\mathcal S(u)\rvert\,\mathrm{d}u.
	\end{align*}
	Since $\mathbb E\smallint_0^1\lvert\mathcal S(u)\rvert\,\mathrm{d}u<\infty$ the dominated convergence theorem shows that $\mathbb E\smallint_0^1\mathbbm 1\big(\lvert\mathcal S(u)\rvert>{N}\big)\lvert\mathcal S(u)\rvert\,\mathrm{d}u\to0$. Therefore there exists an integer $N_\epsilon$ not depending on $A$ such that
	\begin{equation*}
		\mathbb E\int_A\lvert\mathcal S(u)\rvert\,\mathrm{d}u\leq N_\epsilon\lambda(A)+\sqrt{\frac{2\eta_\epsilon(1-\eta_\epsilon)}{\pi}}\cdot\frac{\epsilon}{3}.
	\end{equation*}	
	Consequently
	\begin{equation*}
		\mu_R(A)\leq\sqrt{\frac{\pi}{2\eta_\epsilon(1-\eta_\epsilon)}}\left(N_\epsilon\lambda(A)+\sqrt{\frac{\pi}{2}}\lambda(A)\right)+\frac{2\epsilon}{3}.
	\end{equation*}
	Now if we set
	\begin{equation*}
		\delta_\epsilon=\frac{\sqrt{2\eta_\epsilon(1-\eta_\epsilon)}}{3\sqrt{\pi}(N_\epsilon+\sqrt{\pi/2})}\,\epsilon
	\end{equation*}
	then we have $\mu_R(A)\leq\epsilon$ for every $A\in\mathcal J$ such that $\lambda(A)\leq\delta_\epsilon$. This establishes that $R$ is absolutely continuous.
\end{proof}

\section{Independent samples with different sizes}\label{sec:indep}

The foregoing analysis extends with only minor modifications to the case in which the P--P process is constructed from two independent samples of different sizes. To do this we introduce a nondecreasing function $m:\mathbb N\to\mathbb N$. Instead of defining $F_n$ to be the empirical cdf for $\{X_i\}_{i=1}^n$, define it to be the empirical cdf for $\{X_i\}_{i=1}^{m(n)}$. Assume that $C$ is the product copula. In this setting $R_n$ is the P--P plot based on two independent samples of sizes $m(n)$ and $n$. If we also assume that
\begin{equation*}
	m(n)\to\infty\quad\text{and}\quad\frac{n}{m(n)+n}\to\rho\quad\text{for some }\rho\in[0,1)
\end{equation*}
then Theorem \ref{thm:iff} remains true in this modified setting. The proof of Theorem \ref{thm:necessity} requires only minor adjustments to keep track of the constant $\rho$. The first assertion of Theorem \ref{thm:sufficiency} remains true provided that $\mathcal R$ is defined to be
\begin{equation*}
	\mathcal R(u)=\sqrt{\frac{\rho}{1-\rho}}\mathcal B_1(R(u))-r(u)\mathcal B_2(u),
\end{equation*}
where $\mathcal B_1$ and $\mathcal B_2$ are independent Brownian bridges. In the proof of the first assertion we should define $\tilde{F}_n$ to be the empirical cdf for $\{U_i\}_{i=1}^{m(n)}$, so that
\begin{equation*}
	\sqrt{n}\big((\tilde{F}_n,\tilde{G}_n)-(I,I)\big)\rightsquigarrow\left(\sqrt{\frac{\rho}{1-\rho}}\mathcal B_1,\mathcal B_2\right)
\end{equation*}
in $\ell^\infty[0,1]\otimes\ell^\infty[0,1]$ by Donsker's theorem. An application of the continuous mapping theorem then yields \eqref{eq:donskerjoint} with $\mathcal B_1$ replaced by $\sqrt{\rho/(1-\rho)}\mathcal B_1$. The remainder of the argument is unchanged.

With independent samples the {bootstrap is naturally implemented using independent multinomial weight vectors of dimensions $m(n)$ and $n$} to construct $F_n^\ast$ and $G_n^\ast$. The second assertion of Theorem \ref{thm:sufficiency} remains true with this implementation of the bootstrap and with the modified definition of $\mathcal R$. The proof extends in the same way as the proof of the first assertion.

\end{document}